\documentclass[12pt,a4]{amsart}
\usepackage[a4paper, left=28mm, right=28mm, top=28mm, bottom=34mm]{geometry}
\allowdisplaybreaks
\usepackage{amssymb
,amsthm
,amsmath
,amscd
,mathtools	
,mathdots
,leftidx
,color
}

\usepackage[all]{xy}
\usepackage{url}
\usepackage[dvipdfmx]{graphicx}

\AtBeginDocument{%
  \def\MR#1{}
}

\newcommand{\algO}{\mathrm{O}}

\newcommand\Gm {\mathbb{G}_m}

\newcommand{\oo}{\mathcal{O}}
\newcommand{\p}{\mathfrak{p}}

\newcommand\Z{\mathbb{Z}}

\newcommand\Q{\mathbb{Q}}

\newcommand{\colim}{\mathop{{\mathrm{colim}}}}

\newcommand\et{\textup{\'et}}

\newcommand\Aut{\textup{Aut}}

\newcommand\Gal{\mathrm{Gal}}
\newcommand\Shaf{\mathrm{Shaf}}

\newcommand\Pic{\mathrm{Pic}}

\newcommand\Spec{\mathop{\mathrm{Spec}}}
\newcommand\Spf{\mathop{\mathrm{Spf}}}
\newcommand\chara{\mathop{\mathrm{char}}}
\newcommand\Br{\mathop{\mathrm{Br}}}
\newcommand\Supp{\mathop{\mathrm{Supp}}}
\newcommand\ord{\mathop{\mathrm{ord}}}

\newcommand\m{\mathfrak{m}}

\theoremstyle{plain}

\newtheorem*{lemma*}{Lemma}
\newtheorem*{prop*}{Proposition}
\newtheorem*{theorem*}{Theorem}
\newtheorem*{claim*}{Claim}
\newtheorem{definition*}{Definition}

\theoremstyle{plain}
\newtheorem{theoremsub}[subsection]{Theorem}
\newtheorem{propsub}[subsection]{Proposition}
\newtheorem{lemmasub}[subsection]{Lemma}

\newtheorem{corollarysub}[subsection]{Corollary}

\theoremstyle{definition}

\newtheorem{definitionsub}[subsection]{Definition}
\newtheorem{remarksub}[subsection]{Remark}

\begin{document}
\title{On the Shafarevich conjecture for Enriques surfaces}
\author[T.\ Takamatsu]{Teppei Takamatsu}

\date{\today}


\address{Graduate School of Mathematical Sciences, The University of Tokyo, Komaba, Tokyo, 153-8914, Japan}
\email{teppei@ms.u-tokyo.ac.jp}

\maketitle

\begin{abstract}
Enriques surfaces are minimal surfaces of Kodaira dimension $0$ with $b_{2}=10$. If we work with a field of characteristic away from $2$, Enriques surfaces admit double covers which are K3 surfaces. In this paper, we prove the Shafarevich conjecture for Enriques surfaces by reducing the problem to the case of K3 surfaces. 
In our formulation of the Shafarevich conjecture, we use the notion ``admitting a cohomological good K3 cover'', which includes not only good reduction but also flower pot reduction.
\end{abstract}
\setcounter{section}{0}
\section{Introduction}
The Shafarevich conjecture for abelian varieties states the finiteness of isomorphism classes of abelian varieties of a fixed dimension over a fixed number field admitting good reduction away from a fixed finite set of finite places. The purpose of this paper is to formulate and prove this conjecture for Enriques surfaces. Our main theorem is the following.

\begin{theoremsub}[Theorem \ref{shafenr}]\label{introshafenr}
Let $F$ be a finitely generated field over $\Q$, and $R$ be a finite type algebra over $\Z$ which is a normal domain with the fraction field $F$.
Then, the set 
\[
\left\{X \left|
\begin{array}{l}
X\colon \textup{Enriques surface over }F, \\
\textup{for any height } 1 \textup{ prime ideal } \p \in \Spec R, \\
X \textup{ admits a cohomological good K3 cover at } \p
\end{array}
\right.\right\}
/F \textup{-isom}
\]
is finite.
\end{theoremsub}

Here, we say that $X$ admits a cohomological good K3 cover at $\p$ if there exists a K3 double cover of $X$ whose $\ell$-adic cohomologies are unramified at $\p$ (see Definition \ref{cohom good}). 
In particular, we prove the original Shafarevich conjecture for Enriques surfaces over finitely generated fields over $\Q$.

\begin{corollarysub}[Corollary \ref{corshafenr}]\label{introcorshafenr}
Let $F$ be a finitely generated field over $\Q$, and $R$ be a finite type algebra over $\Z$ which is a normal domain with the fraction field $F$.
Then, the set 
\[
\left\{X \left|
\begin{array}{l}
X\colon \textup{Enriques surface over }F, \\
X \textup{ admits good reduction at any height }1 \textup{ prime ideal }\p \in \Spec R
\end{array}
\right.\right\}
/F\textup{-isom}
\]
is finite.
\end{corollarysub}

Theorem \ref{introshafenr} is, in fact, a direct corollary of the Shafarevich conjecture for K3 surfaces (\cite{She2017}, \cite{Takamatsu2018}). 
To conclude Corollary \ref{introcorshafenr} from Theorem \ref{introshafenr}, we recall the properties of good reduction of Enriques surfaces (see Proposition \ref{ito} and Lemma \ref{goodcohgood}).
We shall make some remarks on the statement of Theorem \ref{introshafenr}. 
It is known that there exists an Enriques surface over a local field which does not admit good reduction, but whose K3 double cover admits good reduction (see \cite[Theorem 3.6]{Liedtke2018}).
Our Theorem \ref{introshafenr} generalizes the original Shafarevich conjecture (Corollary \ref{introcorshafenr}) to deal with such phenomena.

\subsection*{Acknowledgments}
The author is deeply grateful to his advisor Naoki Imai for his deep encouragement and helpful advice. The author also thanks Tetsushi Ito for the proof of Proposition \ref{ito}, Yohsuke Matsuzawa for helpful suggestions, and Yuya Matsumoto for pointing out my misunderstandings about cohomologies of K3 double covers (Remark \ref{remarkcohgood}). 

\section{Enriques surfaces and good reduction}
First, we recall the definition of Enriques surfaces.

\begin{definitionsub}
For any field $k$, an \emph{Enriques surface over $k$} is a smooth projective surface over $k$ satisfying $\omega_{X_{\overline{k}}/\overline{k}} \equiv \oo_{X_{\overline{k}}}$ and $b_{2}(X_{\overline{k}})=10$, where $\equiv$ denotes numerical equivalence.
\end{definitionsub}

If 
$k$ is a field of characteristic away from $2$,
it is well-known that an Enriques surface over $k$ has a K3 double cover (see \cite[Proposition 10.14]{Badescu2001}). 

The main purpose of this paper is to study the Shafarevich conjecture for Enriques surfaces by reducing the problem to the case of K3 surfaces.
In the rest of this section, we will define the good reduction of Enriques surfaces, and study their properties.

\begin{definitionsub}
Let $K$ be a discrete valuation field, and $X$ be an Enriques surface over $K$. We say $X$ admits good reduction if there exists a smooth proper algebraic space $\mathcal{X}$ over $\oo_{K}$ such that $\mathcal{X}_{K}\simeq X$.
\end{definitionsub}

First, we recall the lifting property of line bundles on Enriques surfaces, which was proved by Liedtke.

\begin{lemmasub}\label{lemmaito}
Let $R$ be a complete discrete valuation ring with the residue field $k$. Let $\mathcal{X}\rightarrow \Spf R$ be a formal deformation of Enriques surfaces, i.e. proper smooth formal scheme with the special fiber $\mathcal{X}_{k}$ which is an Enriques surface.
If $\mathcal{L} \in \Pic(\mathcal{X}_{k})$, then $\mathcal{L}^{\otimes 2}$ extends to $\Pic (\mathcal{X})$. 
Moreover, if $H^{2}(\mathcal{X}_{k},\oo) = 0$ (e.g.\  the case of $\chara k \neq 2$), then $\mathcal{L}$ extends to $\Pic (\mathcal{X})$.
\end{lemmasub}
\begin{proof}
This is a special case of \cite[Proposition 4.4 (5)]{Liedtke2015}. We note that Liedtke's argument also works for the case where $k$ is not algebraically closed. For the sake of completeness, we recall the proof.
We put $\mathcal{X} = \colim_{n}\mathcal{X}_{n}$, where $\mathcal{X}_{n}$ is a smooth proper scheme over $S_{n}\coloneqq \Spec (R/\m^{n})$. Here, $\m$ is the maximal ideal of $R$.
Let $\widetilde{R}$ be a strict Henselization of $R$, and $\widetilde{\m}$ be the maximal ideal of $\widetilde{R}$.
Let $\widetilde{\mathcal{X}}_{n} \coloneqq (\mathcal{X}_{n})_{\widetilde{S}_{n}}$ be the base change.
By the Hochschild-Serre spectral sequence with respect to the Galois covering $\widetilde{\mathcal{X}}_{n} \rightarrow \mathcal{X}_{n}$,
we have the exact sequence
\begin{align*}
0 \to H^{1}(G_{k}, H^{0} (\widetilde{\mathcal{X}}_{n}, \oo^{\times})) \to \Pic(\mathcal{X}_{n}) \to \Pic (\widetilde{\mathcal{X}}_{n})^{G_{k}} \to H^{2}(G_{k}, H^{0}(\widetilde{\mathcal{X}}_{n},\oo^{\times})),
\end{align*}
where $G_{k}$ is an absolute Galois group of $k$.
Note that the left end is $0$, and the right end is no other than the Brauer group $\Br (R/\m^{n}) \coloneqq H^{2}_{\et}(\Spec R/\m^{n}, \Gm)$. 
We recall that there exists the natural isomorphisms $\Br (R/\m^{n}) \simeq \Br (k)$
(see \cite[COROLLAIRE 6.2]{Grothendieck1995}, \cite[COROLLAIRE 2.5]{Grothendieck1995a}).
By \cite[Proposition 4.4 (1)-(3)]{Liedtke2015}, for any positive integer $n$, $\mathcal{L}^{\otimes 2}$ lifts to a rational point of the Picard scheme $\mathcal{M}_{n} \in \Pic_{\mathcal{X}_{n}/S_{n}}(S_{n})$, 
which gives an element of $\Pic (\widetilde{\mathcal{X}}_{n})^{G_{k}}$.
By the above exact sequence, the line bundle $\mathcal{M}_{n}$ descends to the line bundle on $\mathcal{X}_{n}$ if and only if its image in $\Br (R/\m^{n})$ is zero.
If $n=1$, this image is zero, thus for any $n$, it is zero. Therefore, $\mathcal{M}_{n}$ is an extension of $\mathcal{L}^{\otimes 2}$ to $\Pic (\mathcal{X})$. 
The final statement follows from \cite[Corollary 8.5.5]{Fantechi2005}.
\end{proof}

Tetsushi Ito taught me the following proposition.

\begin{propsub}\label{ito}
Let $K$ be a complete discrete valuation field, $X$ be an Enriques surface over $K$. The following are equivalent.
\begin{enumerate}
\item
$X$ admits good reduction.
\item
There exists a smooth projective scheme $\mathcal{X}$ over $\oo_{K}$ such that $\mathcal{X}_{K} \simeq X$.
\item
There exists a smooth proper scheme $\mathcal{X}$ over $\oo_{K}$ such that $\mathcal{X}_{K} \simeq X$.
\end{enumerate}
\end{propsub}

\begin{proof}
Clearly $(2)\Rightarrow (3) \Rightarrow (1)$. So we shall prove $(1) \Rightarrow (2)$. Let $k$ be a residue field of $K$. Take a smooth proper algebraic space model $\mathcal{X}$ over $\oo_{K}$. Let $\mathcal{X}_{k}$ be the special fiber of $\mathcal{X}$, and $\widehat{\mathcal{X}}$ be the completion of $\mathcal{X}$ along the special fiber $\mathcal{X}_{k}$. By Lemma \ref{lemmaito}, there exists an ample line bundle of $\mathcal{X}_{k}$ which can be lifted to a line bundle on $\mathcal{X}$. 
Therefore by \cite[Corollary 8.5.6]{Fantechi2005}, we get the desired smooth projective model.
\end{proof}

\begin{remarksub}
As in \cite[Example 5.2]{Matsumoto2015a}, the above proposition does not hold in the case of K3 surfaces.
\end{remarksub}

\section{The Shafarevich conjecture for Enriques surfaces}
First, we recall the definition of a K3 double cover to introduce a notion which generalizes good reduction of Enriques surfaces (see also \cite[Section 7]{Chiarellotto2016}). 

\begin{definitionsub}
\label{K3 double cover}
Let $k$ be a field with characteristic away from $2$, and $X$ a K3 surface over $k$.
By fixing an isomorphism $\omega_{X/k}^{\otimes 2} \simeq \oo_{X}$,
we can equip the $\oo_{X}$-module $\oo_{X} \oplus \omega_{X/k}$ with an $\oo_{X}$-algebra structure.
We refer to a relative spectrum $\widetilde{X} \coloneqq \Spec_{X}(\oo_{X}\oplus \omega_{X/k})$ as a K3 double cover of $X$.
\end{definitionsub}

\begin{definitionsub}
\label{cohom good}
Let $K$ be a discrete valuation field with the residual characteristic $p \neq 2$, and let $X$ be an Enriques surface over $K$.
We say $X$ admits a cohomological good K3 cover if there exists a K3 double cover $\widetilde{X}$ of $X$ such that
$H^{2}_{\et}(\widetilde{X}_{\overline{K}}, \Q_{\ell})$ is an unramified representation, for $\ell \neq p$ (which is independent of $\ell$ as in \cite[Lemma 5.0.1]{Takamatsu2018}).
Yuya Matsumoto taught me that this definition depends on  the choice of $\widetilde{X}$ (see Remark \ref{remarkcohgood}).
\end{definitionsub}

\begin{remarksub} \label{remarkcohgood}
Let $\iota_{1}, \iota_{2} \colon \omega_{X/K}^{\otimes 2} \simeq \oo_{X}$ be isomorphisms satisfying that $\iota_{1} \circ \iota_{2}^{-1} \colon \oo_{X} \rightarrow \oo_{X}$ sends the global section $1$ to $\alpha \in K$.
Let $\widetilde{X}_{i} \rightarrow X$ be K3 double covers defined by $\iota_{1}$ and $\iota_{2}$. 
For $\sigma \in \Gal(\overline{K}/K),$ let $\sigma_{i} \colon \widetilde{X}_{i, \overline{K}} \rightarrow \widetilde{X}_{i, \overline{K}}$ be the Galois action defined by the $K$-structure $\widetilde{X}_{i}$. 
By the definition of a K3 double cover, there exits an isomorphism $\widetilde{X}_{1, K[\sqrt{\alpha}]} \simeq \widetilde{X}_{2,K[\sqrt{\alpha}]}$ over $X_{K[\sqrt{\alpha}]}$. Via this isomorphism, for any $\sigma$ with $\sigma (\sqrt{\alpha}) = - \sqrt{\alpha}$,
$\sigma_{1}$ is equal to $\tau \circ \sigma_{2}$, where $\tau$ is the Enriques involution.
So if the order of $\alpha$ is odd, then at most one of $H^{2}_{\et}(\widetilde{X}_{i}, \Q_{\ell})$  $(i=1,2)$ is unramified. 
We also note that if the order of $\alpha$ is even, then the unramifiedness is preserved.
\end{remarksub}

\begin{lemmasub}\label{goodcohgood}
Let $K$ be a discrete valuation field with the residual characteristic $p \neq 2$, and let $X$ be an Enriques surface over $K$. If $X$ admits good reduction, then 
$X$ admits a cohomological good K3 cover.
\end{lemmasub}

\begin{proof}
Taking a completion, we may assume $K$ is complete (see Remark \ref{remarkcohgood}). 
By Proposition \ref{ito}, one can take a smooth projective scheme model $\mathcal{X}$ over $\oo_{K}$. 
By the separatedness of the Picard functor \cite[Theorem 8.4.3]{Bosch1990}, we have an isomorphism $\omega_{\mathcal{X}/\oo_{K}}^{\otimes2} \simeq \oo_{\mathcal{X}}.$ 
Then we can attach an $\oo_{\mathcal{X}}$-algebra structure on the $\oo_{\mathcal{X}}$-module 
$\oo_{\mathcal{X}} \oplus \omega_{\mathcal{X}/\oo_{K}}$.  
Then a double covering $\widetilde{\mathcal{X}} \coloneqq 
\Spec_{\mathcal{X}} (\oo_{\mathcal{X}} \oplus \omega_{\mathcal{X}/\oo_{K}})$ is finite \'{e}tale over $\mathcal{X}$ by the assumption on the residual characteristic.
So $\widetilde{\mathcal{X}}$ gives a smooth projective scheme model of $\widetilde{\mathcal{X}}_{K}$, which is a K3 double cover of $X$. 
\end{proof}

\begin{remarksub}
There exists an Enriques surface which does not admit good reduction, but whose K3 double cover admits good reduction. See \cite[Theorem 3.6]{Liedtke2018} for such an example having flower pot reduction.
\end{remarksub}

\begin{lemmasub}\label{enrvsk3}
Let $k$ be a field of characteristic away from $2$. Let $Y$ be a K3 surface over $k$.
Then there exists only finitely many $k$-isomorphism classes of  Enriques surfaces 
which is a $\Z/2\Z$-quotient of $Y$.
\end{lemmasub}

\begin{proof}
It suffices to show the finiteness of conjugacy classes of finite order elements in $\Aut_{F}(Y)$. 
By \cite[Proposition 3.10]{Bright2019}, there exists a group homomorphism $\Aut_{F}(Y) \ltimes R_{Y} \rightarrow \algO(\Pic Y)$ which has finite kernel and image of finite index, where $R_{Y}$ is the Galois fixed part of the Weyl group of $Y_{\overline{F}}$.
Here, by \cite[Section 5, (a)]{Borel1963}, there exist only finitely many conjugacy classes of finite order elements in $\algO (\Pic Y)$. 
So the desired property of $\Aut_{F}(Y)$ follows from \cite[Lemma 1.4]{Ohashi2007}.
\end{proof}
%
%

\begin{theoremsub}\label{shafenr}
Let $F$ be a finitely generated field over $\Q$, and $R$ be a finite type algebra over $\Z$ which is a normal domain with the fraction field $F$.
Then, the set 
\[
\Shaf \coloneqq
\left\{X \left|
\begin{array}{l}
X\colon \textup{Enriques surface over }F, \\
\textup{for any height } 1 \textup{ prime ideal } \p \in \Spec R, \\
X \textup{ admits a cohomological good K3 cover at } \p
\end{array}
\right.\right\}
/F \textup{-isom}
\]
is finite.
\end{theoremsub}
\begin{proof}
First, we note that $\Pic (\Spec R)$ is finitely generated (\cite[Chapter 2, Theorem 7.6]{Lang1983}). Let $\oo(D_{1}), \ldots, \oo(D_{r})$ be its generators, where $D_{i}$ are Cartier divisors. Then by removing $\Supp (D_{i})$, we may assume that $\Pic (\Spec R)$ is trivial. Moreover, by shrinking $\Spec R$ again, we may assume $1/2 \in R$.

Next, we will show that for any $X \in \Shaf$, there exists a K3 double cover $\widetilde{X}$ such that for any height $1$ prime $\p \in \Spec R$, $H^{2}_{\et}(X_{\overline{F}}, \Q_{\ell})$ is unramified at $\p$ for $\ell \notin \p$.
Take an element $X \in \Shaf$. Fix an isomorphism $\iota \colon \omega_{X/F}^{\otimes 2} \simeq \oo_{X}$, and let $\widetilde{X}'$ be a corresponding K3 double cover.
One can take a finite set $S$ of height $1$ prime ideals of $\Spec R$ such that for any $\p \notin S$, $H^{2}_{\et}(\widetilde{X}',\Q_{\ell})$ is unramified for $\ell \notin \p$.
Moreover, for any height $1$ prime $\p \in \Spec R$, there exists an isomorphism $\iota_{\p}\colon \omega_{X/F}^{\otimes 2} \simeq \oo_{X}$ such that the corresponding K3 double cover has $\ell$-adic cohomologies which are unramified at $\p$.
Suppose that $\iota_{\p} \circ \iota^{-1} \colon \oo_{X} \rightarrow \oo_{X}$ sends the global section $1$ to $\alpha_{\p} \in F$. We may assume that $\alpha_{\p} = 1$ for $\p \notin S$.
Since $\Pic (\Spec R)$ is trivial, there exists $\alpha \in F$ satisfying that $\ord_{\p} (\alpha) = \ord_{\p} (\alpha_{\p})$ for any height $1$ prime $\p \in \Spec R$.
Then the composition $\alpha \circ \iota \colon \omega_{X/F}^{\otimes 2} \simeq \oo_{X}$ induces a desired K3 double cover $\widetilde{X}$ (see Remark \ref{remarkcohgood}).

Therefore, by Lemma \ref{enrvsk3}, we can reduce the problem to the finiteness of K3 double covers.
So this theorem follows from \cite[Theorem 1.0.1]{Takamatsu2018}.
%
%
\end{proof}

\begin{corollarysub}\label{corshafenr}
Let $F$ be a finitely generated field over $\Q$, and $R$ be a finite type algebra over $\Z$ which is a normal domain with the fraction field $F$.
Then, the set 
\[
\left\{X \left|
\begin{array}{l}
X\colon \textup{Enriques surface over }F, \\
X \textup{ admits good reduction at any height }1 \textup{ prime ideal }\p \in \Spec R
\end{array}
\right.\right\}
/F\textup{-isom}
\]
is finite.
\end{corollarysub}
\begin{proof}
Remark that we may assume $1/2 \in R$.
Therefore, the desired finiteness follows from Lemma \ref{goodcohgood} and Theorem \ref{shafenr}.
\end{proof}


\begin{thebibliography}{FGI{\etalchar{+}}05}

\bibitem[B{\u{a}}d01]{Badescu2001}
Lucian B{\u{a}}descu, \emph{Algebraic surfaces}, Universitext, Springer-Verlag,
  New York, 2001, Translated from the 1981 Romanian original by Vladimir
  Ma\c{s}ek and revised by the author. \MR{1805816}

\bibitem[BLR90]{Bosch1990}
Siegfried Bosch, Werner L\"{u}tkebohmert, and Michel Raynaud, \emph{N\'{e}ron
  models}, Ergebnisse der Mathematik und ihrer Grenzgebiete (3) [Results in
  Mathematics and Related Areas (3)], vol.~21, Springer-Verlag, Berlin, 1990.
  \MR{1045822}

\bibitem[BLvL19]{Bright2019}
Martin Bright, Adam Logan, and Ronald van Luijk, \emph{Finiteness results for
  {K}3 surfaces over arbitrary fields}, European Journal of Mathematics (2019).

\bibitem[Bor63]{Borel1963}
Armand Borel, \emph{Arithmetic properties of linear algebraic groups}, Proc.
  {I}nternat. {C}ongr. {M}athematicians ({S}tockholm, 1962), Inst.
  Mittag-Leffler, Djursholm, 1963, pp.~10--22. \MR{0175901}

\bibitem[CL16]{Chiarellotto2016}
Bruno Chiarellotto and Christopher Lazda, \emph{Combinatorial degenerations of
  surfaces and {C}alabi-{Y}au threefolds}, Algebra Number Theory \textbf{10}
  (2016), no.~10, 2235--2266. \MR{3582018}

\bibitem[FGI{\etalchar{+}}05]{Fantechi2005}
Barbara Fantechi, Lothar G\"{o}ttsche, Luc Illusie, Steven~L. Kleiman, Nitin
  Nitsure, and Angelo Vistoli, \emph{Fundamental algebraic geometry},
  Mathematical Surveys and Monographs, vol. 123, American Mathematical Society,
  Providence, RI, 2005, Grothendieck's FGA explained. \MR{2222646}

\bibitem[Gro95a]{Grothendieck1995}
Alexander Grothendieck, \emph{Le groupe de {B}rauer. {I}. {A}lg\`ebres
  d'{A}zumaya et interpr\'{e}tations diverses}, S\'{e}minaire {B}ourbaki,
  {V}ol. 9, Soc. Math. France, Paris, 1995, pp.~Exp. No. 290, 199--219.
  \MR{1608798}

\bibitem[Gro95b]{Grothendieck1995a}
\bysame, \emph{Le groupe de {B}rauer. {II}. {T}h\'{e}orie cohomologique},
  S\'{e}minaire {B}ourbaki, {V}ol. 9, Soc. Math. France, Paris, 1995, pp.~Exp.
  No. 297, 287--307. \MR{1608805}

\bibitem[Lan83]{Lang1983}
Serge Lang, \emph{Fundamentals of {D}iophantine geometry}, Springer-Verlag, New
  York, 1983. \MR{715605}

\bibitem[Lie15]{Liedtke2015}
Christian Liedtke, \emph{Arithmetic moduli and lifting of {E}nriques surfaces},
  J. Reine Angew. Math. \textbf{706} (2015), 35--65. \MR{3393362}

\bibitem[LM18]{Liedtke2018}
Christian Liedtke and Yuya Matsumoto, \emph{Good reduction of {K}3 surfaces},
  Compos. Math. \textbf{154} (2018), no.~1, 1--35.

\bibitem[Mat15]{Matsumoto2015a}
Yuya Matsumoto, \emph{Good reduction criterion for {K}3 surfaces}, Math. Z.
  \textbf{279} (2015), no.~1-2, 241--266. \MR{3299851}

\bibitem[Oha07]{Ohashi2007}
Hisanori Ohashi, \emph{On the number of {E}nriques quotients of a {$K3$}
  surface}, Publ. Res. Inst. Math. Sci. \textbf{43} (2007), no.~1, 181--200.
  \MR{2319542}

\bibitem[She17]{She2017}
Yiwei She, \emph{{The unpolarized Shafarevich Conjecture for K3 Surfaces}},
  preprint (2017), arXiv:1705.09038.

\bibitem[Tak18]{Takamatsu2018}
Teppei Takamatsu, \emph{{On a cohomological generalization of the Shafarevich
  conjecture for K3 surfaces}}, preprint (2018), arXiv:1810.12279.

\end{thebibliography}

\newcommand{\etalchar}[1]{$^{#1}$}
\providecommand{\bysame}{\leavevmode\hbox to3em{\hrulefill}\thinspace}
\providecommand{\MR}{\relax\ifhmode\unskip\space\fi MR }
\providecommand{\MRhref}[2]{%
  \href{http://www.ams.org/mathscinet-getitem?mr=#1}{#2}
}
\providecommand{\href}[2]{#2}

\end{document}